\pdfoutput=1
\documentclass{amsart}
\usepackage[longtable]{multirow}
\usepackage{graphicx}
\usepackage{longtable}
\usepackage{amsmath}
\usepackage{amsthm}
\usepackage{amsfonts}
\usepackage{mathtools}
\usepackage{fullpage}
\usepackage{tikz}
\usepackage{amssymb}
\newtheorem{theorem}{Theorem}
\newtheorem{observation}{Observation}
\newtheorem{fact}{Fact}
\newtheorem{lemma}{Lemma}
\newtheorem{corollary}{Corollary}
\newtheorem{claim}{Claim}
\theoremstyle{definition}
\newtheorem*{definition}{Definition}
\newtheorem*{remark}{Remark}
\newtheorem{example}{Example}

\usepackage{rotating}
\usepackage{accents}
\usepackage[margin=1in]{geometry}
\usetikzlibrary{arrows}
\usetikzlibrary{arrows.meta}
\tikzset{edge/.style={->,> = latex'}}
\tikzset{
	DD/.style={double distance=5pt, shorten >=00pt,-{Classical TikZ Rightarrow[length=3mm]}},
	Rightarrow/.style={line width=1pt, double distance=4.5pt, -{Classical TikZ Rightarrow[length=3mm]}},
	triple/.style={preaction={draw,Rightarrow},-},
	quadruple/.style={preaction={draw,Rightarrow,shorten >=0pt},shorten >=3.5pt,-,double,double distance=1.5pt}
}

\title{Koszul modules of Kac-Moody Lie algebras}
\thanks{The author is supported by the grant MAESTRO NCN - UMO-2019/34/A/ST1/00263
	- Research in Commutative Algebra and Representation Theory.}
\author{Tymoteusz Chmiel}
\address{Jagiellonian University,
	ul. {\L}ojasiewicza 6,
	30-348 Krak\'ow,
	Poland}

\begin{document}
\maketitle

\begin{abstract}
We introduce Koszul modules associated with (graded) Kac-Moody Lie algebras. We provide a precise criterion for when these modules are of finite length. As an exemplary application we deduce a bound on the dimension of the second graded component for a certain class of graded Kac-Moody Lie algebras. We also provide an exact description of all nilpotent Kac-Moody Koszul modules.
\end{abstract}

\section*{Introduction}

Koszul modules are graded modules over a polynomial ring $S=\mathbb{C}[x_1\cdots,x_n]$ defined in terms of the Koszul complex resolving the trivial $S$-module $\mathbb{C}$. They were introduced in \cite{Papadima-Suciu}, where they were used to give a unified proof of the vanishing of the resonance of outer Torelli groups of surface groups. In \cite{AFPRW1, AFPRW2} a uniform bound on the length of \textit{nilpotent} Koszul modules was obtained. This result turned out to have many interesting applications to the study of invariants associated with finitely generated groups \cite{AFPRW1}. It also provided a new proof of the generic Green's conjecture \cite{AFPRW2}. Later in \cite{Raicu-Sam} a generalization to a bi-graded setting resulted in the proof of the Canonical Ribbon Conjecture. Numerous applications of the theory of Koszul modules in algebraic geometry can be found in \cite{AFRW}.

The goal of this paper is to present yet another setting in which Koszul modules appear. In general, a Koszul module is associated with a pair $(V,K)$, where $V$ is a finite-dimensional vector space and $K\subset\bigwedge^2 V$ is a subspace. If $\mathfrak{g}=\bigoplus_{i\geq 0}\mathfrak{g}_i$ is a graded Lie algebra, a natural Koszul module associated with $\mathfrak{g}$ is given by setting $V:=\mathfrak{g}_1$ and taking $K$ to be the kernel of the Lie bracket. We investigate a situation when $\mathfrak{g}$ is (the positive part) of a Kac-Moody Lie algebra, with the grading induced by a choice of a simple root. Such gradings were recently found to be important in construction of generic resolutions of length $3$ \cite{Weyman}.

When it comes to applications, particularly important are \emph{nilpotent} Koszul modules, i.e. those with only finitely many non-zero graded components. Our main result is Theorem \ref{th:1}, which provides a simple condition equivalent to the nilpotency of a given Kac-Moody Koszul module. The proof uses only elementary facts about Kac-Moody Lie algebras, together with a theorem of Papadima-Suciu characterizing equivariant Koszul modules of finite length. Furthermore, using the Kostant's formula, we are able to provide a precise intrinsic description of all nilpotent Kac-Moody Koszul modules (Theorem \ref{th:max}).

The structure of this paper is as follows. In section 1 we define Koszul modules and associated resonance varieties. Section 2 provides the definition of Kac-Moody Lie algebras together with their basic properties. Section 3 introduces Koszul modules associated with Kac-Moody Lie algebras. Main results are contained in the last two section: in section 4 we prove that the nilpotency of a Kac-Moody Koszul module is equivalent to a simple condition on the associated generalized Cartan matrix, while section 5 gives a precise description of all nilpotent Koszul modules which arise this way.

\section{Koszul modules and resonance varieties}

Let $V$ be a finite dimensional vector space over $\mathbb{C}$ and let $S:=\text{Sym}(V)\cong\mathbb{C}[x_1,\cdots,x_n]$ be the symmetric algebra on $V$. The minimal free resolution of a trivial $S$-module $\mathbb{C}$ is the \textit{Koszul complex}
$$\mathcal{K}:0\rightarrow S\otimes\bigwedge^nV\xrightarrow{\delta_n}\cdots\xrightarrow{\delta_{p+1}}S\otimes\bigwedge^pV\xrightarrow{\delta_p}\cdots\xrightarrow{\delta_2}S\otimes V\xrightarrow{\delta_1}S,$$
where the differentials on basis elements are given by the formula
$$\delta_p(v_1\wedge\cdots\wedge v_p)=\displaystyle\sum_{i=1}^{p}v_i\otimes (v_1\wedge\cdots\wedge\hat{v_i}\wedge\cdots\wedge v_p)$$

\begin{definition}
Let $\iota:K\hookrightarrow\bigwedge^2V$ be a linear subspace. The \textit{Koszul module} $\mathcal{W}(V,K)$ is the $S$-module with the presentation given by
$$S\otimes(\bigwedge^3V\oplus K)\xrightarrow{\delta_3\oplus(\text{id}\otimes\iota)}S\otimes\bigwedge^2V\twoheadrightarrow\mathcal{W}(V,K)$$
\end{definition}

\noindent If we put $\bigwedge^2V$ and $K$ in degree $0$ and $\bigwedge^3V$ in degree $1$, the Koszul module $\mathcal{W}(V,K)$ inherits this grading. Its graded pieces are given by
\begin{equation}\label{eq:gr}
\mathcal{W}(V,K)_q=\textnormal{cokernel}\left\{\textnormal{Sym}^{q-1}\otimes\bigwedge^3V\rightarrow\textnormal{Sym}^q\otimes\left(\bigwedge^2V/K\right)\right\}
\end{equation}

Several properties of Koszul modules follow directly from the fact that the Koszul complex is a minimal free resolution. For example, one can see that an alternative way to define $\mathcal{W}(V,K)$ is as the cokernel $\textnormal{im}(\delta_2)\big/\textnormal{im}(\delta_2\circ\iota)$. Furthermore, $\mathcal{W}(V,K)=0$ if and only if $K=\bigwedge^2V$. On the other extreme we have $\mathcal{W}(V,\{0\})_q=H^0\left(\mathbb{P}(V^*),\Omega^1(q+2)\right)$. Note also that if we have an inclusion $K'\subset K$, then there is a surjective map of Koszul modules $\mathcal{W}(V,K')\twoheadrightarrow\mathcal{W}(V,K)$. In fact it is just a simple example of a much stronger functoriality \cite{Papadima-Suciu}.

As with any graded module, a fundamental question about Koszul modules is whether they have \textit{finite length}, i.e. if the graded pieces eventually vanish: $\mathcal{W}(V,K)_q=0$ for $q\gg0$. If the graded pieces are finite-dimensional, it is of course equivalent to the fact that $\dim\mathcal{W}(V,K)<\infty$. A graded module over a polynomial ring has finite length if and only if it is \textit{nilpotent}, i.e. its support vanishes.

An important \textit{geometric} object for studying the nilpotency of a Koszul module $\mathcal{W}(V,K)$ is the associated \textit{resonance variety} $\mathcal{R}(V,K)$. Let $K^\perp:=\{a\in\bigwedge^2V^*:a\big|_K\equiv0\}$. Then the resonance variety is defined as
$$
\mathcal{R}(V,K)=\{a\in V^*:\exists b\in V^*\textnormal{ such that }  0\neq a\wedge b\in K^{\bot}\}\cup\{0\}
$$

These resonance varieties were introduced by Papadima and Suciu in \cite{Papadima-Suciu}, where they proved that away from the origin $\mathcal{R}(V,K)$ equals the support of the Koszul module $\mathcal{W}(V,K)$. Thus the module $\mathcal{W}(V,K)$ is nilpotent if and only if $\mathcal{R}(V,K)=\{0\}$ (see Theorem \ref{th:PS}). Using this description they showed that the set of subspaces $K$ with $\dim K=m$ and $\mathcal{R}(V,K)=\{0\}$ is Zariski-open inside the Grassmanian $\textnormal{Gr}_m(\bigwedge^2V)$ and that it is non-empty exactly for $m\geq 2n-3$. From this one can conclude that there exists a number $q:=q(n,m)$ such that if $\dim V=n$, $\dim K=m$ and $\mathcal{W}(V,K)$ is nilpotent, then $\mathcal{W}(V,K)_q=0$. This was made even more explicit in \cite{AFPRW1}, \cite{AFPRW2} where authors showed that one can take $q=n-3$. Thus one can always check whether the Koszul module $\mathcal{W}(V,K)$ is nilpotent by computing the cokernel (\ref{eq:gr}) for $q=n-3$.

We will consider a situation where $V$ is a representation of a semi-simple Lie algebra $\mathfrak{g}$ and $K\subset\bigwedge^2V$ is a $\mathfrak{g}$-submodule. Then the graded components $\mathcal{W}(V,K)_q$ are representations of $\mathfrak{g}$ as well, and the resonance variety $\mathcal{R}(V,K)$ is a $\mathfrak{g}$-invariant subset of $V^*$. If $\mathfrak{g}$ acts on $V$ with finitely many orbits, this allows one to compute the resonance variety explicitly as a sum of these orbits via simple linear algebra. Therefore in this case the verification of whether the Koszul module is nilpotent can be accomplished very easily. In general the following result of Papadima and Suciu provides a very useful criterion for determining whether an equivariant Koszul module is nilpotent. In this theorem we assume that both the representation $V$ and the Lie algebra $\mathfrak{g}$ are finite-dimensional.

\begin{theorem}\label{th:PS}
Let $V$ be an irreducible representation of a semi-simple Lie algebra $\mathfrak{g}$ with the highest weight $\lambda$ and let $K\subset\bigwedge^2 V$ be a subrepresentation. Let $V^*:=V(\lambda^*)$ be the dual representation and put $K^\perp:=\{a\in\bigwedge^2 V^*:a|_K\equiv 0\}$.

The following conditions are equivalent:
\begin{itemize}
\item $\mathcal{R}(V,K)=\{0\}$;
\item the Koszul module $\mathcal{W}(V,K)$ is nilpotent;
\item $2\lambda^*-\beta$ is not a dominant weight of $K^\perp$ for any simple root $\beta$.
\end{itemize}	
\end{theorem}

\section{Kac-Moody Lie algebras}

A natural context in which Koszul modules appear is that of \emph{graded Lie algebras}. Let $\mathfrak{g}$ be a Lie algebra and let $\mathfrak{g}=\bigoplus_{i\in\mathbb{Z}}\mathfrak{g}_i$ be a direct sum decomposition such that $[\mathfrak{g}_i,\mathfrak{g}_j]\subset\mathfrak{g}_{i+j}$. The Lie bracket is linear and skew-symmetric and thus defines a map $[\cdot,\cdot]:\bigwedge^2\mathfrak{g}_1\rightarrow\mathfrak{g}_2$. Now if $\mathfrak{g}_1$ is finite-dimensional, we can associate with $\mathfrak{g}$ the Koszul module $\mathcal{W}(\mathfrak{g}_1,K)$, where $K$ is the kernel of $[\cdot,\cdot]$. Note that since $[\mathfrak{g}_0,\mathfrak{g}_i]\subset\mathfrak{g}_{i}$ for each $i\in\mathbb{Z}$, the graded pieces $\mathfrak{g}_i$ are representations of the Lie subalgebra $\mathfrak{g}_0\subset\mathfrak{g}$. Hence the same is true for $\mathcal{W}(\mathfrak{g}_1,K)$ and it provides an example of an equivariant Koszul module. 

We will be interested in a very particular class of graded Lie algebras: \emph{Kac-Moody Lie algebras}. They include all finite-dimensional Lie algebras (classical and exceptional) and provide their natural generalization to the infinite-dimensional setting. Here we will only provide an overview of the construction of Kac-Moody Lie algebras which will be sufficient for our purposes; for details the reader may consult \cite{Kac}.

Let $A=(a_{i,j})_{1\leq i,j\leq n}$ be a complex $n\times n$ matrix satysfying the following conditions:
\begin{itemize}
\item $a_{i,i}=2$
\item if $i\neq j$, then $a_{i,j}\in\mathbb{Z}_{\leq 0}$
\item $a_{i,j}=0\implies a_{j,i}=0$
\end{itemize}	
Such a matrix is called a \emph{generalized Cartan matrix}. With such matrix we also associate its \emph{generalized Dynkin diagram} which is a convenient way of encoding all the necessary information. The generalized Dynkin diagram of $A$ is a directed graph with multiple edges. It has nodes $\alpha_1,\cdots,\alpha_n$ corresponding to the columns of $A$ and the following edges between each two nodes:

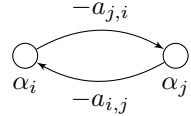
\begin{center}
\begin{tikzpicture}
\node [circle,fill=white,draw,label=below:$\alpha_i$] (1) at (0,0) {};
\node [circle,fill=white,draw,label=below:$\alpha_j$] (2) at (2,0) {};
\draw[edge] (1) to [bend left] node[midway,above]{$-a_{j,i}$} (2);
\draw[edge] (2) to [bend left] node[midway,below]{$-a_{i,j}$} (1);
\end{tikzpicture} 
\end{center}
We do not draw edges which would be labelled with $0$ and sometimes replace arrows of multiplicity $1,2$ and $3$ with arrows composing of an appropriate number of lines.

Let $A$ be a generalized Cartan matrix. A \textit{realization} of $A$ is a vector space $\mathfrak{h}$ together with subsets $H^\vee=\{\alpha^\vee_1,\cdots,\alpha^\vee_n\}\subset\mathfrak{h}$ and $H=\{\alpha_1,\cdots,\alpha_n\}\subset\mathfrak{h}^*$. These objects are required to satisfy the following conditions: $\dim\mathfrak{h}=n+\text{corank}(A)$, both $H$ and $H^*$ are linearly independent and $\alpha_i(\alpha_j^\vee)=a_{i,j}$. It is not hard to see that a realization of $A$ always exists and that it is unique up to an isomorphism.

To define a Kac-Moody Lie algebra $\mathfrak{g}(A)$ we begin by defining an auxiliary Lie algebra $\widetilde{\mathfrak{g}}(A)$. It is generated by the vector space $\mathfrak{h}$ in the realization of $A$ and two sets of generators $e_i,f_i$ for $i=1,\cdots,n$. On these generators we impose the following relations:
\begin{itemize}
\item $[\mathfrak{h},\mathfrak{h}]=\{0\}$;
\item $[e_i,f_i]=\delta_{i,j}\alpha^\vee_i$;
\item $[h,e_i]=\alpha_i(h)\cdot e_i$ for all $h\in\mathfrak{h}$;
\item $[h,f_i]=-\alpha_i(h)\cdot f_i$ for all $h\in\mathfrak{h}$.
\end{itemize}
The Lie algebra $\widetilde{\mathfrak{g}}(A)$ contains the unique maximal ideal $\mathfrak{r}$ such that $\mathfrak{r}\cap\mathfrak{h}=\{0\}$. Now we define the \textit{Kac-Moody Lie algebra} associated with $A$ to be $\mathfrak{g}(A):=\widetilde{\mathfrak{g}}(A)/\mathfrak{r}$.

Kac-Moody Lie algebras $\mathfrak{g}(A)$ are naturally divided into three rather different classes:
\begin{itemize}
	\item if $\det A>0$, $\mathfrak{g}(A)$ is of finite type;
	\item if $\det A=0$, $\mathfrak{g}(A)$ is of affine type;
	\item if $\det A<0$, $\mathfrak{g}(A)$ is of indefinite type.
\end{itemize}
Kac-Moody Lie algebras of finite type are exactly those for which $\dim\mathfrak{g}(A)<\infty$. Thus they are precisely the well-known finite-dimensional, semi-simple Lie algebras (classical and exceptional). Among indefinite matrices we distinguish a class of \textit{hyperbolic type}. These are matrices for which each proper submatrix is of finite or affine type. We will use the phrase 'of ... type' to describe all of the three related objects: the matrix $A$, the generalized Dynkin diagram $D$ and the Lie algebra $\mathfrak{g}(A)=\mathfrak{g}(D)$. Thus for example a diagram is of hyperbolic type when all of its proper subdiagrams are of finite or affine type.

One of the most important features of Lie algebras $\widetilde{\mathfrak{g}}(A)$ and $\mathfrak{g}(A)$ is their direct sum decomposition with respect to the $\mathfrak{h}$-eigenvalues. Let $(\mathfrak{h},H^\vee,H)$ be a realization of $A$. The elements of $H$ are called \textit{simple roots} and the lattice spanned by $H\subset\mathfrak{h}^*$ is called the \textit{root lattice} $Q$. We denote by $Q_+$, resp. $Q_-$ the semigroups generated by $H$, resp. $-H$. Then we have direct sum decompositions:
\begin{equation}\label{direct_sum_decomp}
\widetilde{\mathfrak{g}}(A)=\left(\displaystyle\bigoplus_{\alpha\in Q_-}\widetilde{\mathfrak{g}}_{\alpha}\right)\oplus\mathfrak{h}\oplus\left(\displaystyle\bigoplus_{\alpha\in Q_+}\widetilde{\mathfrak{g}}_{\alpha}\right)\quad\textnormal{and}\quad
\mathfrak{g}(A)=\left(\displaystyle\bigoplus_{\alpha\in Q_-}\mathfrak{g}_{\alpha}\right)\oplus\mathfrak{h}\oplus\left(\displaystyle\bigoplus_{\alpha\in Q_+}\mathfrak{g}_{\alpha}\right),
\end{equation}
where for any element $g\in\widetilde{\mathfrak{g}}_\alpha\cup\mathfrak{g}_\alpha$ we have $[h,g]=\alpha(h)g$. Elements of the root lattice $\alpha\in Q$ such that $\mathfrak{g}_\alpha\neq\{0\}$ are called \textit{roots of} $\mathfrak{g}(A)$.

The $\mathbb{Z}$-grading $\bigoplus_{i\in\mathbb{Z}}\mathfrak{g}_i$ on a Kac-Moody Lie algebra $\mathfrak{g}:=\mathfrak{g}(A)$ is a coarser version of the root space decomposition (\ref{direct_sum_decomp}). Let $\alpha_1,\cdots,\alpha_n$ be the simple roots of $\mathfrak{g}(A)$ and pick a distinguished root $\alpha:=\alpha_1$. Any positive root $\delta$ can be uniquely written as $\sum a_j\alpha_j$ for some non-negative integers $a_i$, $i=1,\cdots,n$. Let us define
$$Q_+^{\alpha}(i):=\{\delta\in Q_+: \delta=\sum a_j\alpha_j\textnormal{ with }a_1=i\}$$
Then the graded pieces $\mathfrak{g}_i$ of $\mathfrak{g}(A)$ are defined as
$$
\mathfrak{g}_i:=\bigoplus_{\delta\in Q_+^{\alpha}(i)}\mathfrak{g}_\delta\quad\textnormal{for any}\quad i\in\mathbb{Z}
$$
This grading has many properties which are direct consequences of the root space decomposition. For example the duality $\mathfrak{g}_i\simeq\left(\mathfrak{g}_{-i}\right)^*$ follows immediately from the fact that the roots of $\mathfrak{g}$ are symmetric about the origin.

The following simple lemmas will be of use later:

\begin{lemma}\label{l:connected}
	
	Let $\delta=\sum a_i\alpha_i$ be a root of a Kac-Moody Lie algebra $\mathfrak{g}$. The subgraph of the Dynkin diagram consisting of the nodes corresponding to the simple roots $\alpha_i$ with $a_i\neq0$ is connected.
\end{lemma}

\begin{proof} This is Lemma 1.6 from \cite{Kac}.
\end{proof}

\begin{lemma}\label{l:generating} Let $\delta=\sum a_i\alpha_i$ be a root of a Kac-Moody Lie algebra $\mathfrak{g}$ and let $\mathfrak{g}'$ be the Lie subalgebra generated by the root spaces $\mathfrak{g}_{\pm\alpha_i}$ for all indices $i$ such that $a_i\neq0$. Then $\mathfrak{g}_\delta\subset\mathfrak{g}'$. \end{lemma}

\begin{proof} Let $\alpha_1,\cdots,\alpha_n$ be the simple roots of $\mathfrak{g}$ and $e_1,\cdots,e_n$ the corresponding generators. Directly from the definition of a Kac-Moody Lie algebra it follows that for any positive root $\delta$ the eigenspace $\widetilde{\mathfrak{g}_\delta}$ inside $\widetilde{\mathfrak{g}}$ is generated by the elements $[e_{i_1},[e_{i_2},[\cdots,[e_{i_{N-1}},e_{i_N}]\cdots]]]$ with $\alpha_{i_1}+\cdots+\alpha_{i_N}=\delta$. Hence the same is true for the Lie algebra $\mathfrak{g}$. The claim follows.\end{proof}

\begin{lemma}\label{l:subdiagram}
	Let $D'$ be a subgraph of a Dynkin diagram $D$ and let $\mathfrak{g}'\subset\mathfrak{g}(D)$ be the Lie subaglebra generated by root spaces $\mathfrak{g}_{\pm\alpha}$ for simple roots $\alpha$ corresponding to the nodes of $D'$. Then $\mathfrak{g}'=\mathfrak{g}(D')$.
\end{lemma}

\begin{proof}
	This is Exercise 1.9 from \cite{Kac}.
\end{proof}

\section{Koszul modules of Kac-Moody Lie algebras}

Since we want to study the Koszul module associated to a graded Lie algebra $\mathfrak{g}$, we are particularly interested in the low degree graded pieces $\mathfrak{g}_0,\ \mathfrak{g}_1$ and $\mathfrak{g}_2$. Let us start with the Lie subalgebra $\mathfrak{g}_0$. Clearly it contains the Cartan subalgebra $\mathfrak{h}$ and the root spaces $\mathfrak{g}_{\pm\alpha_i}$ for $i\neq1$. To recover its precise form, we use Lemma \ref{l:subdiagram}. It implies that the Lie algebra $\mathfrak{g}_0$ is a direct sum of a Kac-Moody Lie algebra $\mathfrak{g}\left(D\setminus\{\alpha\}\right)$ and an abelian Lie algebra $\mathbb{C}^r$, where $r=2$ if $D$ is affine, $r=0$ if $D\setminus\{\alpha\}$ is affine and $r=1$ otherwise. Since we want to employ representation-theoretic methods to the study of Koszul modules, we must require that the Lie algebra $\mathfrak{g}_0$ is finite-dimensional. It is a classical fact that Kac-Moody Lie algebra $\mathfrak{g}(D)$ is finite-dimensional if and only if its generalized Dynkin diagram is of finite type. This motivates the following definition:

\begin{definition}
Let $D$ be a generalized Dynkin diagram, let $\alpha$ be a node and let $D'$ be the subdiagram of $D$ obtained by removing $\alpha$. We call a pair $(D,\alpha)$ \emph{admissible} if each connected component of $D'$ is of finite type.
\end{definition}

Now assume that $(D,\alpha):=(D,\alpha_1)$ is an admissible pair. We want to understand $\mathfrak{g}_1$ as a representation of $\mathfrak{g}_0$. Recall that $\mathfrak{g}_1$ is the dual of $\mathfrak{g}_{-1}$. It is not hard to see that the standard generator $f_1$ of $\mathfrak{g}$, which satisfies $\mathfrak{g}_{-\alpha_1}=\mathbb{C}f_1$, is the unique highest weight vector for $\mathfrak{g}_{-1}$. Thus the graded component $\mathfrak{g}_{-1}$ as a $\mathfrak{g}_0$-module is isomorphic to $V(\alpha_1\big|_{\mathfrak{h}_0})$, where $\mathfrak{h}_0$ is the Cartan subalgebra of $\mathfrak{g}_0$. The weight $\alpha_1\big|_{\mathfrak{h}_0}$ corresponds to the first column of the generalized Cartan matrix $A$ associated with $D$, except of course with the first entry removed. Case-by-case analysis reveals that we have:

\begin{fact}\label{fact}
Let $A=\left(a_{i,j}\right)_{i,j=0,\cdots,n}$ be a generalized Cartan matrix and let $D$ be the associated generalized Dynkin graph.
Assume that $(D,\alpha_0)$ is an admissible pair. Let $\oplus_{i\in\mathbb{Z}}\mathfrak{g}_i$ be the graded Kac-Moody Lie algebra associated with this pair. Denote by $\mathfrak{g}_0^{ss}$ the semi-simple part of $\mathfrak{g}_0$ and let $\omega_1,\cdots,\omega_n$ be its fundamental weights. Then as an $\mathfrak{g}^{ss}_0$-module $\mathfrak{g}_1$ is isomorphic to $V(-\sum a_{i,0}\omega_i)$.	
\end{fact}	

Recall that each finite-dimensional, irreducible $\mathfrak{g}^{ss}_0$-module is isomorphic to a unique representation $V(\sum a_{i}\omega_i)$ for some $a_i\in\mathbb{Z}_{\geq 0}$. Thus for every semi-simple Lie algebra $\mathfrak{k}$ and every irreducible representation $V$, both assumed to be finite-dimensional, there exists a Kac-Moody Lie algebra $\mathfrak{g}=\mathfrak{g}(A)$ and a simple root $\alpha$ such that for the associated grading $\mathfrak{k}$ is the semi-simple part of $\mathfrak{g}_0$ and $\mathfrak{g}_1\simeq V$. Indeed, let $A'=\left(a_{i,j}\right)_{i,j=1,\cdots,n}$ be the Cartan matrix of $\mathfrak{k}$ and let $V\simeq V(\sum a_{i}\omega_i)$ as above. Now let $D$ be the generalized Dynkin diagram of a generalized Cartan matrix $A=\left(a_{i,j}\right)_{i,j=0,\cdots,n}$, where $a_{i,0}:=-a_i$, $a_{0,i}:=0$ if $a_{i,0}=0$, while $a_{0,i}\in\mathbb{Z}_{<0}$ for the remaining indices. Then $(D,\alpha_0)$ is an admissible pair such that for the associated grading $\mathfrak{g}_0^{ss}\simeq\mathfrak{k}$ and $\mathfrak{g}_1\simeq V$. Of course this pair is highly non-unique.

The following easy lemma basically shows that the Lie subalgebra $\oplus_{i>0}\mathfrak{g}_i$ is generated by $\mathfrak{g}_1$.

\begin{lemma}\label{l:surjective} Let $\mathfrak{g}$ be a Kac-Moody Lie algebra and let $\bigoplus_{i\geq0}\mathfrak{g}_i$ be the grading associated with some simple root $\alpha$. Then the Lie brakcet map $[\cdot,\cdot]:\bigwedge^2\mathfrak{g}_1\rightarrow\mathfrak{g}_2$ is surjective. \end{lemma}

\begin{proof} Recall that the root space $\mathfrak{g}_\delta$ is generated by the elements $[e_{i_1},[e_{i_2},[\cdots,[e_{i_{N-1}},e_{i_N}]\cdots]]]$ such that $\alpha_{i_1}+\cdots+\alpha_{i_N}=\delta$. Hence $\mathfrak{g}_2$ is also generated by elements of this form, more precisely those for which index $1$ appears exactly twice amongst $i_1,\cdots,i_N$. Let $i_k$ be the first index equal to $1$ and define $g:=[e_{i_{k+1}},[\cdots,[e_{i_{N-1}},e_{i_N}]\cdots]]\in\mathfrak{g}_1$, so that $[e_{i_1},[e_{i_2},[\cdots,[e_{i_{N-1}},e_{i_N}]\cdots]]]=[e_{i_1},[e_{i_2},[\cdots,[e_{i_{k-1}},[e_1,g]]\cdots]]]$. We will show that if $a,b\in\mathfrak{g}_1$, then $[e_{i_1},[e_{i_2},[\cdots,[e_{i_{k-1}},[a,b]]\cdots]]]\in\mathfrak{g}_2$ is in the image of $[\cdot,\cdot]$. If $k=1$, then there is nothing to prove. Assume now $k\geq2$. Using the Jacobi identity we obtain \begin{align*} &[e_{i_1},[e_{i_2},[\cdots,[e_{i_{k-1}},[a,b]]\cdots]]]=\\ &=-[e_{i_1},[e_{i_2},[\cdots,[e_{i_{k-2}},[a,[b,e_{i_{k-1}}]+[b,[e_{i_{k-1}},a]]\cdots]]]=\\ &=-[e_{i_1},[e_{i_2},[\cdots,[e_{i_{k-2}},[a,[b,e_{i_{k-1}}]]]]\cdots]]-[e_{i_1},[e_{i_2},[\cdots,[e_{i_{k-2}},[b,[e_{i_{k-1}},a]]]\cdots]]]\\ \end{align*} \vspace{0pt} and the claim follows by induction. \end{proof}

If $(D,\alpha)$ is an admissible pair, $\mathfrak{g}_1$ is a finite-dimensional irreducible representation of $\mathfrak{g}_0$ and it makes sense to talk about the associated Koszul module. This way we arrive at the following central definition:

\begin{definition}
Let $(D,\alpha)$ be an admissible pair, let $\mathfrak{g}:=\mathfrak{g}(D)$ and let $\bigoplus_{i\geq 0}\mathfrak{g}_i$ be the grading associated with the simple root corresponding to $\alpha$. The Koszul module associated to $(D,\alpha)$ is defined as $$\mathcal{W}(D,\alpha):=\mathcal{W}\left(\mathfrak{g}_1,\ker\left([\cdot,\cdot]:\bigwedge^2\mathfrak{g}_1\rightarrow\mathfrak{g}_2\right)\right)$$.
\end{definition}

The natural question is when the Koszul module $\mathcal{W}(D,\alpha)$ is nilpotent.

\begin{example}
	Consider the Kac-Moody Lie algebras associated with the following generalized Cartan matrices:
	$$
	\begin{pmatrix}
	2&-2\\-1&2\\
	\end{pmatrix}\text{ and }
	\begin{pmatrix}
	2&-2\\-2&2\\
	\end{pmatrix}
	$$
	\smallskip
	Their corresponding Dynkin diagrams are $D_1$ and $D_2$ respectively:
	
	
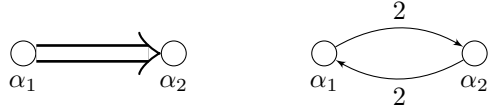
\begin{center}
		\begin{tikzpicture}
		\node [circle,fill=white,draw,label=below:$\alpha_1$] (1) at (0,0) {};
		\node [circle,fill=white,draw,label=below:$\alpha_2$] (2) at (2,0) {};
		\draw[line width=1pt, double distance=5pt, -{Classical TikZ Rightarrow[length=3mm]}] (1) to (2);
		\node [circle,fill=white,draw,label=below:$\alpha_1$] (3) at (4,0) {};
		\node [circle,fill=white,draw,label=below:$\alpha_2$] (4) at (6,0) {};
		\draw[edge] (3) to [bend left] node[midway,above]{2} (4);
		\draw[edge] (4) to [bend left] node[midway,below]{2} (3);
		\end{tikzpicture}
	\end{center}
	Of course both pairs $(D_1,\alpha_1)$ and $(D_2,\alpha_1)$ are admissible and in both cases $\mathfrak{g}_0^{ss}\simeq\mathfrak{sl}(2)$. Similarly, the formula above shows that $\mathfrak{g}_1\simeq V(2\omega_1)=\textnormal{Sym}^2\mathbb{C}^2$ in both cases as well. However, in the case of $D_1$ the graded piece $\mathfrak{g}_2$ vanishes, while for $D_2$ we have $\mathfrak{g}_2\simeq\textnormal{Sym}^2\mathbb{C}^2$ and in fact all graded pieces are of this form. Thus we see that the Koszul module $\mathcal{W}(D_1,\alpha_1)$ is nilpotent, while the Koszul module $\mathcal{W}(D_2,\alpha_1)$ is not.
\end{example}

\section{Characterization of nilpotent pairs}

\begin{theorem}\label{th:1}
	Let $(D,\alpha)$ be an admissible pair and let $m$ be the maximal multiplicity of an arrow directed towards the node $\alpha$. The Koszul module $\mathcal{W}(D,\alpha)$ is nilpotent if and only if $m=1$.
\end{theorem}

\begin{proof}
Put $\mathfrak{g}:=\mathfrak{g}(D)$, $\mathfrak{k}:=\ker\left([\cdot,\cdot]:\bigwedge^2\mathfrak{g}_1\rightarrow\mathfrak{g}_2\right)$ and $\mathfrak{k}^\perp:=\{a\in\bigwedge^2\mathfrak{g}^*_1: a\big |_\mathfrak{k}\equiv0\}$.
Let $A=(a_{i,j})_{i,j=0,\cdots,n}$ be the generalized Cartan matrix associated with $D$ and let $\alpha_0,\alpha_1,\cdots,\alpha_n$ be the simple roots of $\mathfrak{g}$. We assume that $\alpha=\alpha_0$ and then we identify $\alpha_1,\cdots,\alpha_n$ with the simple roots of $\mathfrak{g}^{ss}_0$. Note that the condition on the multiplicity of edges adjacent to $\alpha$ is equivalent to $|a_{0,i}|\leq 1$ for $i=1\cdots,n$.

Let $\lambda$ be the highest weight of $\mathfrak{g}_1$ as a $\mathfrak{g}^{ss}_0$-module, so that $\mathfrak{g}_1\simeq V(\lambda)$. We will use Theorem \ref{th:PS} as our nilpotency criterion: the Koszul module $\mathcal{W}(D,\alpha)$ is nilpotent if and only if $2\lambda^*-\alpha_i$ is not a dominant weight of $\mathfrak{k}^\perp$ for any $i=1,\cdots,n$. Recall that $V(\lambda^*)\simeq\mathfrak{g}_1^*\simeq\mathfrak{g}_{-1}$ and hence $\mathfrak{k}^\perp\subset\bigwedge^2\mathfrak{g}_{-1}$. The lowest weight vector of $\mathfrak{g}_1$ is clearly the standard generator $e_0$, hence the highest weight vector of $\mathfrak{g}_{-1}$ is $f_0$. Since the $\mathfrak{g}$-weight of the vector $f_0$ is $-\alpha_0$, we conclude that $\lambda^*=-\alpha_0\big|_{\mathfrak{h}_0}$, where $\mathfrak{h}_0$ is the Cartan subalgebra of $\mathfrak{g}^{ss}_0$. Thus we have to show that in the generalized Cartan matrix we have $|a_{0,i}|\leq 1$ for all $i=1\cdots,n$ if and only if $-2\alpha_0\big|_{\mathfrak{h}_0}-\alpha_i$ is not a dominant weight of $\mathfrak{k}^\perp$ for any $i=1,\cdots,n$.

Note that the multiplicity of an isomorphic copy of $V(2\lambda^*-\alpha_i)$ inside $\bigwedge^2\mathfrak{g}_{-1}$ is at most one. Indeed, the highest weight vector of this copy must be an element of the at most one-dimensional space $\mathfrak{g}_{-2\alpha_0-\alpha_i}=\mathbb{C}\cdot\left[f_0,[f_0,f_i]\right]$. Furthermore, by Lemma \ref{l:surjective} the map $[\cdot,\cdot]:\bigwedge^2\mathfrak{g}_{-1}\rightarrow\mathfrak{g}_{-2}$ is surjective. It follows that $-2\alpha_0\big|_{\mathfrak{h}_0}-\alpha_i$ is not a dominant weight of $\mathfrak{k}^\perp$ if and only if it is not a dominant weight of $\mathfrak{g}_{-2}$. Since weights of this form will always be dominant, we conclude that $-2\alpha_0\big|_{\mathfrak{h}_0}-\alpha_i$ is not a dominant $\mathfrak{g}^{ss}_0$-weight for $\mathfrak{k}^\perp$ if and only if $-2\alpha_0-\alpha_i$ is a not root of $\mathfrak{g}$. By the symmetry of the root system, the second condition is of course equivalent to $2\alpha_0+\alpha_i$ not being a root of $\mathfrak{g}$.

Let us now fix $i\in\{1,\cdots,n\}$ and put $\delta:=2\alpha_0+\alpha_i$. Let $\mathfrak{g}'\subset\mathfrak{g}$ be the Lie subalgebra generated by the elements $e_0,e_i,f_0,f_i$. Allowing for a slight abuse of notation, we treat $\alpha_0,\  \alpha_1$ and $\delta$ as roots of both $\mathfrak{g}$ and $\mathfrak{g}'$. By Lemma \ref{l:generating} we know that if $\delta$ is a root of $\mathfrak{g}$, then $\mathfrak{g}_\delta\subset\mathfrak{g}'$. Hence $\delta$ is a root of $\mathfrak{g}$ if and only if it is a root of $\mathfrak{g}'$. On the other hand, by Lemma \ref{l:subdiagram} the Lie algebra $\mathfrak{g}'$ is isomorphic to the rank $2$ Kac-Moody Lie algebra $\mathfrak{g}(A)$ for the generalized Cartan matrix $$A=\begin{pmatrix}2&a_{0,i}\\a_{i,0}&2\\\end{pmatrix}
=\begin{pmatrix}2&-a\\-b&2\\\end{pmatrix}$$
Using Lemma \ref{l:connected} to exclude the trivial case $a=b=0$, we have thus reduced our theorem to the following:

\begin{claim}\label{claim}
Let $a,b\in\mathbb{Z}_{>0}$ and let $D$ be the generalized Dynkin diagram

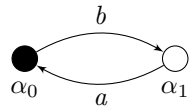
\begin{center}
\begin{tikzpicture}
\node [circle,fill=black,draw,label=below:$\alpha_0$] (1) at (0,0) {};
\node [circle,fill=white,draw,label=below:$\alpha_1$] (2) at (2,0) {};
\draw[edge] (1) to [bend left] node[midway,above]{b} (2);
\draw[edge] (2) to [bend left] node[midway,below]{a} (1);
\end{tikzpicture}
\end{center}
Then the $2\alpha_0+\alpha_1$ is \textnormal{not} a root if and only if $a=1$.
\end{claim}

An instructive example is given by pairs $(a,b)$ for which the generalized Dynkin diagram $D$ is of finite type. In Table \ref{table:finite} we present the corresponding root systems and one immediately concludes that for those Lie algebras $2\alpha_0+\alpha_1$ is a root if and only if $\alpha_0$ is a short root, i.e. $a\neq 1$.

\begin{center}
	\begin{table}[htb]
		\centering
		\[
		\begin{array}{|c|c|c|c|}
		\hline\rule[-0.5mm]{0mm}{4mm}
		\textup{$(a,b)$}
		&
		$(1,1)$
		&
		$(1,2)$
		&
		$(1,3)$
		\\\hline\rule[8mm]{0mm}{5mm}
		
		\textup{\Large{Dynkin diagram}}
		&
		\begin{tikzpicture}
		\node [circle,fill=white,draw,label=below:$\alpha_0$] (1) at (0,0) {};
		\node [circle,fill=white,draw,label=below:$\alpha_1$] (2) at (2,0) {};
		\draw (1) to (2);
		\end{tikzpicture}
		&
		\begin{tikzpicture}
		\node [circle,fill=white,draw,label=below:$\alpha_0$] (1) at (0,0) {};
		\node [circle,fill=white,draw,label=below:$\alpha_1$] (2) at (2,0) {};
		\draw[line width=1pt, double distance=4.5pt, -{Classical TikZ Rightarrow[length=3mm]}] (1) to (2);
		\end{tikzpicture}
		&
		\begin{tikzpicture}
		\node [circle,fill=white,draw,label=below:$\alpha_0$] (1) at (0,0) {};
		\node [circle,fill=white,draw,label=below:$\alpha_1$] (2) at (2,0) {};
		\draw[triple] (1) to (2);
		\end{tikzpicture} 
		\\\hline
		
		\textup{\Large{Type}}
		&
		\textup{$A_2$}
		&
		\textup{$B_2=C_2$}
		&
		\textup{$G_2$}
		\\\hline
		
		\textup{\Large{Lie algebra}}
		&
		\textup{$\mathfrak{sl}(3)$}
		&
		\textup{$\mathfrak{so}(5)=\mathfrak{sp}(4)$}
		&
		\textup{$\mathfrak{g}_2$}
		\\\hline\rule[10mm]{0mm}{5mm}
		
		\textup{\Large{Root system}}	
		&
		\begin{tikzpicture}
		\foreach\ang in {60,120,...,360}{
			\draw[->,black!80!black,thick] (0,0) -- (\ang:1cm);}
		\node[right,scale=1.5] at (1,0) {$\alpha_1$};
		\node[above,scale=1.5] at (-0.5,0.9) {$\alpha_0$};
		\end{tikzpicture}
		&
		\begin{tikzpicture}
		\foreach\ang in {90,180,270,360}{
			\draw[->,black!80!black,thick] (0,0) -- (\ang:1cm);
		}
		\foreach\ang in {45,135,225,315}{
			\draw[->,black!80!black,thick] (0,0) -- (\ang:1.4142cm);
		}
		
		\node[right,scale=1.5] at (1,0) {$\alpha_1$};
		\node[above right,scale=1.5] at (-1.4,1) {$\alpha_0$};
		\end{tikzpicture}
		&
		\begin{tikzpicture}
		\foreach\ang in {60,120,...,360}{
			\draw[->,black!80!black,thick] (0,0) -- (\ang:1cm);	}
		\foreach\ang in {30,90,...,330}{
			\draw[->,black!80!black,thick] (0,0) -- (\ang:1.6cm);	}
		\node[right,scale=1.5] at (1,0) {$\alpha_1$};
		\node[above right,scale=1.5] at (-1.7,0.8) {$\alpha_0$};
		\end{tikzpicture}
		\\\hline
		\end{array}
		\] 
		\caption{Finite-dimensional rank 2 Kac-Moody Lie algebras}\label{table:finite}
	\end{table}
\end{center}

Now we prove Claim \ref{claim} in general. Let $\widetilde{\mathfrak{g}}:=\widetilde{\mathfrak{g}}(D)$, $\mathfrak{g}:=\mathfrak{g}(D)$ and let $\widetilde{\mathfrak{i}}$, resp. $\mathfrak{i}$, denote the ideal generated by the root space $\widetilde{\mathfrak{g}}_{2\alpha_0+\alpha_1}$ inside $\widetilde{\mathfrak{g}}$, resp. by $\mathfrak{g}_{2\alpha_0+\alpha_1}$ inside $\mathfrak{g}$. Recall that $\widetilde{\mathfrak{g}}_{2\alpha_0+\alpha_1}=\mathbb{C}\cdot g$, where $g:=\left[e_0,[e_0,e_1\right]]$. By usual abuse of notation we treat $g$ as an element of both $\widetilde{\mathfrak{g}}$ and $\mathfrak{g}$, and $g\neq0$ inside $\mathfrak{g}$ if and only if $2\alpha_0+\alpha_1$ is a root.

Using the defining relations of $\widetilde{\mathfrak{g}}$ and the Jacobi identity, we compute:
\begin{align*}\label{ideal}
&[f_0,g]=[f_0,[e_0,[e_0,e_1]]]=-[[e_0,e_1],[e_0,f_0]]-[e_0,[f_0,[e_0,e_1]]]=\\
&[\alpha_0^\vee,[e_0,e_1]]+[e_0,[e_0,[e_1,f_0]]]+[e_0,[e_1,[f_0,e_0]]]=\\
&-[e_0,[e_1,\alpha_0^\vee]]-[e_1,[\alpha_0^\vee,e_0]]+[e_0,[\alpha_0^\vee,e_1]]=\\
&2(1-a)\cdot[e_0,e_1]\\
\end{align*}

Assume first that $a\neq1$. In this case the computation above directly shows that $[e_0,e_1]\in\widetilde{\mathfrak{i}}$. Hence $\widetilde{\mathfrak{g}}_{\alpha_1}=\mathbb{C}e_1=\mathbb{C}\cdot[f_0,[e_0,e_1]]\subset\widetilde{\mathfrak{i}}$ and consequently $\alpha_1^\vee=[e_1,f_1]\in\widetilde{\mathfrak{i}}\cap\mathfrak{h}$. Recall the ideal $\mathfrak{t}$ from the definition of $\mathfrak{g}$. It is the unique maximal ideal inside $\widetilde{\mathfrak{g}}$ intersecting $\mathfrak{h}$ trivially and $\mathfrak{g}=\widetilde{\mathfrak{g}}/\mathfrak{t}$. Since $\widetilde{\mathfrak{i}}\cap\mathfrak{h}\neq\{0\}$ and $\widetilde{\mathfrak{g}}_{2\alpha_0+\alpha_1}$ generates $\widetilde{\mathfrak{i}}$, it follows that $\widetilde{\mathfrak{g}}_{2\alpha_0+\alpha_1}\cap\mathfrak{t}=\{0\}$. Thus inside $\mathfrak{g}$ we have $\mathfrak{g}_{2\alpha_0+\alpha_1}\neq\{0\}$.

Now assume that $a=1$. By the computation above $[f_0,g]=0$, while $[f_1,g]=0$ due to the fact that $2\alpha_0$ is not even a root of $\widetilde{\mathfrak{g}}$. Consider the subspace $\mathfrak{s}\subset\mathfrak{g}$ consisting of linear combinations of the elements of the form $[e_{i_1},[e_{i_2},[\cdots,[e_{i_N},g]]\cdots]]$ with $i_k\in\{0,1\}$ and $N\geq 0$. We claim that the subspace $\mathfrak{s}$ is equal to $\mathfrak{i}$. We obviously have $g\in\mathfrak{s}\subset\mathfrak{i}$, so we only need to show that $\mathfrak{s}$ itself is an ideal. The element $g$ is an eigenvector of the action of $\mathfrak{h}$, and so are the elements $[e_{i_1},[e_{i_2},[\cdots,[e_{i_N},g]]\cdots]]$. Hence $[\mathfrak{h},\mathfrak{s}]\subset\mathfrak{s}$. Inclusions $[\mathfrak{g}_{\alpha_0},\mathfrak{s}]\subset\mathfrak{s}$ and $[\mathfrak{g}_{\alpha_1},\mathfrak{s}]\subset\mathfrak{s}$ follow directly from the definition of $\mathfrak{s}$. Finally, combining the Jacobi identity with the fact that $[f_0,g]=[f_1,g]=0$ implies $[\mathfrak{g}_{-\alpha_0},\mathfrak{s}]\subset\mathfrak{s}$ and $[\mathfrak{g}_{-\alpha_1},\mathfrak{s}]\subset\mathfrak{s}$. Since $\mathfrak{h},e_0,e_1,f_0$ and $f_1$ generate $\mathfrak{g}$, $\mathfrak{s}$ is indeed an ideal. However, $\mathfrak{s}\cap\mathfrak{h}=\{0\}$ and consequently $\mathfrak{s}=\{0\}$. Thus $\mathfrak{g}_{2\alpha_0+\alpha_1}=\{0\}$ as well.

\end{proof}	

\begin{remark}
	An even quicker but less direct way to prove Claim \ref{claim} is to instead consider the grading on $\mathfrak{g}$ defined by the root $\alpha_1$. Then we have $\left\langle f_0, \alpha_0^\vee,e_0 \right\rangle=\mathfrak{g}_0^{ss}\simeq\mathfrak{sl}(2)$ and under this identification the root $\alpha_0$ corresponds to the weight $2\omega$, where $\omega$ is the fundamental weight of $\mathfrak{sl}(2)$. Recall that $f_1$ is the highest weight vector of the representation $\mathfrak{g}_{-1}$. The duality of the grading, self-duality of $\mathfrak{sl}(2)$-modules and Fact \ref{fact} all together imply that as representations of $\mathfrak{sl}(2)$ we have $\mathfrak{g}_1\simeq(\mathfrak{g}_{-1})^*\simeq\mathfrak{g}_{-1}\simeq V(a\omega)=\textnormal{Sym}^a(\mathbb{C}^2)$. Weights of the representation $\textnormal{Sym}^a(\mathbb{C}^2)$ form an uninterrupted string $-a\cdot\omega,-(a-2)\cdot\omega,\cdots,(a-2)\cdot\omega,a\cdot\omega$ and the weight $2\alpha_0+\alpha_1$ corresponds to $4\cdot\omega-a\cdot\omega=(4-a)\cdot\omega$. Thus it is a root of $\mathfrak{g}$ if and only if $-a\leq 4-a\leq a$, i.e. when $2\leq a$.
\end{remark}

\begin{remark}
Yet another proof of Claim \ref{claim} can be given using the following direct description of root systems for a certain class of Kac-Moody Lie algebras:

\begin{theorem}[Proposition 5.10 in \cite{Kac}]\label{th:roots}
	Let $A$ be a symmetrizable generalized Cartan matrix of finite, affine or hyperbolic type. Denote by $(\cdot,\cdot)$ the invariant bilinear form on $\mathfrak{h}\simeq\mathfrak{h}^*$ and let $\Delta_+\subset Q_+$ be the set of positive roots of $\mathfrak{g}(A)$.
	
	Then $\Delta_+=\Delta_+^{re}\cup\Delta_+^{im}$, where
	$$\Delta^{re}_+:=\{\delta=\sum a_i\alpha_i\in Q_+: (\delta,\delta)>0\textnormal\quad{ and}\quad a_i\cdot\frac{(\alpha_i,\alpha_i)}{(\delta,\delta)}\in\mathbb{Z}\}$$
	$$\Delta^{im}_+:=\{\delta=\sum a_i\alpha_i\in Q_+: (\delta,\delta)\leq0\}$$
	The elements of $\Delta^{re}_+$ (resp. $\Delta^{im}_+$), are \textnormal{real} (resp. \textnormal{imaginary}) \textnormal{roots} of $\mathfrak{g}(A)$.
\end{theorem}

Every rank $2$ generalized Cartan matrix is of finite, affine or hyperbolic type. Indeed, the only subdiagrams of $D$ correspond to the finite root system $A_1$. Furthermore, clearly such matrix is symmetrizable. Thus Theorem \ref{th:roots} holds for $\mathfrak{g}(D)$. An invariant bilinear form $(\cdot,\cdot)$ on $\mathfrak{h}\simeq\mathfrak{h}^*$ is given by the symmetrization of the generalized Cartan matrix $A$. Thus in the situation at hand it can be written as:
$$
\begin{pmatrix}
2b&-ab\\
-ab&2a\\
\end{pmatrix}
$$
\noindent Consider an element $\delta=m\alpha_0+n\alpha_1$ for some $m,n\in\mathbb{Z}_{\geq0}$. The norm of $\delta$ with respect to $(\cdot,\cdot)$ is then given by $(\delta,\delta)=2bm^2-2abmn+2an^2$. The root we are interested in is $\delta:=2\alpha_0+\alpha_1$ and hence $(\delta,\delta)=8b+2a-4ab$.

If $a=1$, this gives $(\delta,\delta)=4b+2>0$. Thus $\delta\not\in\Delta_+^{im}$. To see that $\delta\not\in\Delta_+^{re}$, it is enough to note that $2\cdot\tfrac{(\alpha_0,\alpha_0)}{(\delta,\delta)}=2\cdot\tfrac{2b}{4b+2}=1-\frac{1}{2b+1}\not\in\mathbb{Z}$. Thus $\delta\not\in\Delta$.

For $a=2$, we have $(\delta,\delta)=4$ and since $(\alpha_0,\alpha_0)=2b$, $(\alpha_1,\alpha_1)=4$, we conclude that $\delta\in\Delta^{re}_+$. For $a=3$ we get $(\delta,\delta)=6-4b$ and so $\delta\in\Delta^{im}_+$ unless $b=1$. But for $b=1$ we get $(\delta,\delta)=2$ and hence $\delta\in\Delta^{re}_+$. Now if $a\geq 4$, $(\delta,\delta)=8b+2a-4ab\leq 8b+2ab-4ab=2b(4-a)\leq 0$. Thus in this case $\delta$ is always an imaginary root.
\end{remark}

Recall that in the previous section we have presented a construction of an admissible pair $(D,\alpha)$ with prescribed $\mathfrak{g}_0^{ss}$ for which the representation $\mathfrak{g}_1$ could be made isomorphic to any irreducible, finite-dimensional $\mathfrak{g}_0^{ss}$-module. The generalized Dynkin diagram $D$ was associated with the generalized Cartan matrix $A=(a_{i,j})_{i,j=0,\cdots,n}$ with all entries uniquely specified, except for $a_{0,i}\in\mathbb{Z}_{<0}$ for $1\leq i\leq n$ such that $a_{i,0}\neq 0$. Theorem \ref{th:1} shows that for the module $\mathcal{W}(D,\alpha)$ to be nilpotent we must have $a_{0,i}=-1$ for all unspecified indices. Thus we obtain the following:

\begin{corollary}
	Let $\mathfrak{k}$ be a finite-dimensional, semi-simple Lie algebra, let $\omega_1,\cdots,\omega_n$ be its fundamental weights and let $V=V(\sum a_i\omega_i)$ be an irreducible representation of $\mathfrak{k}$. Then there exists \textnormal{unique} admissible pair $(D,\alpha)$ such that for the associated graded Kac-Moody Lie algebra $\mathfrak{g}=\mathfrak{g}(D)$ we have $\mathfrak{g}_0^{ss}\simeq\mathfrak{k}$, $\mathfrak{g}_1\simeq V$ and the Koszul module $\mathcal{W}(D,\alpha)$ is nilpotent.
\end{corollary}

Using Theorem \ref{th:1} we also get a non-trivial bound on the dimension of $\mathfrak{g}_2$ in terms of $\dim\mathfrak{g}_1$. Note that by Lemma \ref{l:generating} the trivial bound is $\dim\mathfrak{g}_2\leq\binom{\dim\mathfrak{g}_1}{2}$ and the equality is achieved quite often, e.g. for all symmetric rank $2$ Kac-Moody Lie algebras other than $\mathfrak{sl}(2)$.

\begin{corollary}
Let $D$ be a generalized Dynkin diagram with a node $\alpha$ such that no multiple arrow points towards $\alpha$. Then for the associated grading on the Lie algebra $\mathfrak{g}=\mathfrak{g}(D)$ we have an inequality: $$\dim\mathfrak{g}_2\leq\binom{\dim\mathfrak{g}_1}{2}-(2\dim\mathfrak{g}_1-3)$$
\end{corollary}

\begin{proof}
If the pair $(D,\alpha)$ is not admissible, there is nothing to prove. Assume that $(D,\alpha)$ is an admissible pair. By Theorem \ref{th:1} we know that the Koszul module $\mathcal{W}(D,\alpha)$ is nilpotent. By Proposition 2.7 from \cite{Papadima-Suciu} a Koszul module $\mathcal{W}(V,K)$ can be nilpotent only when $\dim K\geq 2\dim V-3$. In the definition of $\mathcal{W}(D,\alpha)$ we take $V=\mathfrak{g}_1$ and $K=\ker\left([\cdot,\cdot]\right)$. Since the map $[\cdot,\cdot]:\bigwedge^2\mathfrak{g}_1\rightarrow\mathfrak{g}_2$ is surjective, $\dim K=\binom{\dim\mathfrak{g}_1}{2}-\dim\mathfrak{g}_2$ and the claim follows.
\end{proof}	

\section{Maximal nilpotent modules}

In this section we provide a precise description of \textit{nilpotent} Koszul modules $\mathcal{W}(D,\alpha)$ for all admissible pairs $(D,\alpha)$. It turns out that they can be characterized purely in terms of the defining representation $\mathfrak{g}_1$. Unlike results in the previous section whose proofs were given using only elementary methods, the present section uses homology of Lie algebras, the Kostant's formula and the explicit description of the graded pieces $\mathfrak{g}_{-j}$ provided by \cite{Kang}.

Let $V=V(\lambda)$, $\dim V<\infty$, be an irreducible representation of a finite-dimensional semi-simple Lie algebra $\mathfrak{g}$. Theorem \ref{th:PS} implies that there exists a unique $\mathfrak{g}$-invariant subspace $K_{max}\subset \bigwedge^2V$ such that the Koszul module $\mathcal{W}(V,K_{max})$ is \textit{maximal nilpotent}, i.e. for any $\mathfrak{g}$-invariant subspace $K\subset \bigwedge^2V$ such that $\mathcal{W}(V,K)$ is nilpotent there exists a natural surjective $\mathfrak{g}$-equivariant map $\mathcal{W}(V,K_{max})\twoheadrightarrow\mathcal{W}(V,K)$.

Indeed, since \textit{smaller} subspaces correspond to \textit{larger} Koszul module, it is enough to find the smallest subspace satisfying any of the conditions of Theorem \ref{th:PS}. That such a subspace exists and is unique is a direct consequence of the third condition: it is enough to assure that the \textit{lowest} weights of $K_{\max}$ are precisely $(2\lambda^*-\beta)^*$ for simple roots $\beta$ such that $(\lambda,\beta)\neq 0$.

We will denote the unique maximal nilpotent Koszul module on $V$ by by $\mathcal{W}_{max}(V)$.

\begin{theorem}\label{th:max}
Let $(D,\alpha)$ be an admissible pair. If the Koszul module $\mathcal{W}(D,\alpha)$ is nilpotent, then it is maximal nilpotent: $\mathcal{W}(D,\alpha)=\mathcal{W}_{max}(\mathfrak{g}_1)$.
\end{theorem}

As already mentioned, the proof of this theorem is a direct application of results from \cite{Kang}. This paper provides a formula for the characters of the graded pieces $\mathfrak{g}_{-j}$ of a graded Kac-Moody Lie algebra $\mathfrak{g}$ in terms of the homology of the Lie subalgebra $\mathfrak{g}_-:=\bigoplus_{i=0}^{+\infty}\mathfrak{g}_{-i}$.

Let $\mathbb{C}$ be the trivial $\mathfrak{g}_0$-module. The homology modules $H_i(\mathfrak{g}_-)$ are defined to be the homology groups of the complex of $\mathfrak{g}_0$-modules:
$$
\cdots\xrightarrow{d_{k+1}}\bigwedge^k\mathfrak{g_-}\xrightarrow{d_k}\bigwedge^{k-1}\mathfrak{g_-}\xrightarrow{d_{k-1}}\cdots\xrightarrow{d_3}\bigwedge^2\mathfrak{g_-}\xrightarrow{d_2}\bigwedge^1\mathfrak{g_-}\xrightarrow{d_1}\bigwedge^0\mathfrak{g}_-\xrightarrow{d_0}\mathbb{C}\rightarrow 0
$$
with the differentials given by
$$d_k(g_1\wedge\cdots\wedge g_k):=\displaystyle\sum_{1\leq i<j\leq k}(-1)^{i+j}[g_i,g_j]\wedge g_1\wedge\cdots\wedge\hat{g_i}\wedge\cdots\wedge\hat{g_j}\wedge\cdots\wedge g_k$$
\noindent These homology groups can be computed using the \emph{Kostant's formula}. To state this formula we must introduce a couple of definitions.

The \textit{Weyl group} of $\mathfrak{g}$ is the subgroup $W\subset\operatorname{Aut}(\mathfrak{h}^*)$ generated by \textit{simple reflections}: $r_i.f:=f-f(\alpha_i^\vee)\alpha_i$ for all simple positive roots $\alpha_i$. Note that $r_i^2=\operatorname{id}$. The \textit{length} of an element $w\in W$ is the length of a minimal representation of $w$ using the generators $r_i$ and we denote it by $l(w)$. The lengths are bounded (equivalently the Weyl group is finite) if and only if $\dim(\mathfrak{g})<+\infty$. Let $\Delta^+$ be the positive roots of $\mathfrak{g}$ and let $\Delta^+_0\subset\Delta^+$ be the positive roots of $\mathfrak{g}_0$. For $w\in W$ let $\Phi_w:=\{\alpha\in\Delta^+: w^{-1}.\alpha<0\}$ and let $W_0:=\{w\in W:\Phi_w\subset\Delta^+\setminus\Delta^+_0\}$. Finally let $\rho\in\mathfrak{h}^*$ be the \emph{Weyl vector} defined by $\rho(\alpha_i^\vee)=1$ for all $i$.

\begin{theorem}[Kostant's formula, \cite{GL}, \cite{Liu}] We have the following decomposition into irreducibles $\mathfrak{g}_0$-modules:
$$
H_k(\mathfrak{g_-})\simeq\displaystyle\bigoplus_{w\in W_0, l(w)=k}V(w.\rho-\rho)
$$
\end{theorem}

Note that the $\mathbb{Z}$-grading on the Kac-Moody Lie algebra $\mathfrak{g}$ induces a $\mathbb{Z}$-grading on the homology modules $H_k(\mathfrak{g_-})$. We will combine the Kostant's formula with the following result which is an immediate corollary of Theorem 3.1 from \cite{Kang}:

\begin{observation}\label{ob:g-2}
$$
\mathfrak{g}_{-2}\simeq\frac{\bigwedge^2\mathfrak{g}_{-1}}{H_2(\mathfrak{g_-})_{-2}}
$$
\end{observation}

Now we are ready to proof Theorem \ref{th:max}.

\begin{proof}
As mentioned at the beginning of this section, to prove that $\mathcal{W}(D,\alpha)=\mathcal{W}_{max}(\mathfrak{g}_1)$ it is enough to show that if $\lambda$ is the lowest weights of $K=\ker\left([\cdot,\cdot]\right)$, then it is of the form $2\alpha+\beta$ for some simple positive root $\beta$. Since $\mathfrak{g}_i\simeq(\mathfrak{g}_{-i})^*$ this is equivalent to showing that if a highest weights of $\bigwedge^2\mathfrak{g}_{-1}$ is \textit{not} of the from $-(2\alpha+\beta)$, it is also a highest weight of $\mathfrak{g}_{-2}$. Using the Observation \ref{ob:g-2} we see that if $\lambda$ is a highest weight of $\bigwedge^2\mathfrak{g}_{-1}$ which is not a highest weight of $H_2(\mathfrak{g_-})_{-2}$, then $V(\lambda)$ is a factor of $\mathfrak{g}_{-2}$. Thus it is enough to show that all the highest weights of $H_2(\mathfrak{g_-})_{-2}$ are of the from $-(2\alpha+\beta)$.

From the Kostant's formula the highest weights of $H_2(\mathfrak{g_-})$ are of the form $w.\rho-\rho$ for certain elements of the Weyl group $w\in W$ satisfying $l(w)=2$. Let $w=r_ir_j$, $i\neq j$. Note that we have $r_i.\alpha_j=\alpha_j+n\alpha_i$ for some integer $n\in\mathbb{Z}$. Thus we compute
$$
w.\rho=r_i.(r_j.\rho)=r_i.(\rho-\rho(\alpha^\vee_j)\alpha_j)=r_i.(\rho-\alpha_j)=r_i.\rho-r_i\alpha_j=
$$
$$=(\rho-\alpha_i)-(\alpha_j+n\alpha_i)=\rho-\big((n+1)\alpha_i+\alpha_j\big)
$$
Hence $w.\rho-\rho=-\big((n+1)\alpha_i+\alpha_j\big)$. To land in the graded piece $H_2(\mathfrak{g_-})_{-2}$ it is clearly necessary that $\alpha_i=\alpha$ and $n+1=2$. Thus $w.\rho-\rho=-(2\alpha+\alpha_j)$ which concludes the proof.

\end{proof}

\end{document}